\documentclass[11pt,twoside]{amsart}
\usepackage[T1]{fontenc}
\usepackage[utf8]{inputenc}
\usepackage[english]{babel}
\usepackage{graphicx}
\usepackage{amsmath}
\usepackage{amssymb}
\usepackage{enumerate}
\usepackage[all,cmtip]{xy}
\usepackage{url}
\usepackage{mathrsfs}

\usepackage{amsthm}
\usepackage{amsfonts}
\usepackage{changepage}

\newtheorem{thm1}{Theorem}[section]
\newtheorem{theorem}[thm1]{Theorem}
\newtheorem{lemma}[thm1]{Lemma}
\newtheorem{corollary}[thm1]{Corollary}
\newtheorem{proposition}[thm1]{Proposition}

\newtheorem{thmx}{Theorem}

\theoremstyle{definition}
\newtheorem{definition}[thm1]{Definition}

\newtheorem{paragraf}[thm1]{}

\newcommand{\blap}[1]{\vbox to 7pt{\hbox{#1}\vss}}
\newcommand{\bblap}[1]{\vbox to 7pt{\hbox{#1}\vss}}

\RequirePackage{ifthen}
\RequirePackage{calc}
\newcounter{exampleendflag}
\newcommand{\exendhere}{
  \setcounter{exampleendflag}{0} 
  \ifmmode
    \eqno
    \ensuremath{\blacktriangle}
  \else
    \hspace{\stretch{1}}
    \ensuremath{\blacktriangle}
  \fi
}
\newenvironment{example}{
  \setcounter{exampleendflag}{1}
  \begin{exx}
}{
  \ifthenelse{\value{exampleendflag}=1}{\exendhere}{} 
  \end{exx}
} 

\newenvironment{ex}{
  \setcounter{exampleendflag}{1}
  \begin{exx}
}{
  \ifthenelse{\value{exampleendflag}=1}{\exendhere}{} 
  \end{exx}
} 
\newtheorem{exx}[thm1]{Example}

\theoremstyle{remark}
\newtheorem{remark}[thm1]{Remark}

\title{Gotzmann's persistence theorem for finite modules}
\author{\gss}%
\thanks{The author is supported by the Swedish Research Council, grant number 2011-5599.}
\newcommand{\saeden}{S\ae d\'en}
\newcommand{\stahl}{St\aa hl}
\newcommand{\gss}{Gustav \saeden\ \stahl}

\newcommand{\proj}{\operatorname{Proj}}

\newcommand{\fO}{\mathcal{O}}
\newcommand{\x}{\mathbf{x}}
\renewcommand{\i}{\mathbf{i}}
\renewcommand{\j}{\mathbf{j}}

\renewcommand{\phi}{\varphi}
\renewcommand{\epsilon}{\varepsilon}

\newcommand{\C}{\mathbb{C}}

\renewcommand{\P}{\mathbb{P}}

\newcommand{\Fitt}{\operatorname{Fitt}}

\numberwithin{equation}{section}
\address{Department of Mathematics, KTH Royal Institute of Technology, SE-100 44 Stockholm, Sweden}
\email{gss@math.kth.se}
\begin{document}
\begin{abstract}
We prove a generalization of Gotzmann's persistence theorem in the case of modules with constant Hilbert polynomial. As a consequence, we show that the defining equations that give the embedding of a Quot scheme of points into a Grassmannian are given by a single Fitting ideal.
\end{abstract}
\maketitle

\section*{Introduction}
One of the main results of this paper is the following.
\begin{thmx}\label{thm:a}
Let $A$ be a ring, let $S=A[X_0,\dots,X_r]$, let $M=\bigoplus_{i=1}^pS$, and let $N$ be a graded $S$-submodule of $M$, generated in degrees at most $d$. Write $Q=M/N$ and let $n\le d$. If $Q_t$ is locally free of rank~$n$ for $t=d$ and $t=d+1$, then $Q_t$ is locally free of rank~$n$ for all~${t\ge d}$.
\end{thmx}
This theorem concerns homogeneous submodules $N\subseteq \bigoplus_{i=1}^pA[X_0,\dots,X_r]$ generated in degrees at most $d$, for some $d$. A special case of such a submodule is a homogeneous ideal ${I\subseteq A[X_0,\dots,X_r]}$ generated in degrees at most $d$. In that case, when $A$ is noetherian, we have Gotzmann's persistence theorem \cite{MR0480478} which states that if the graded component $Q_t$ of the quotient ${Q=A[X_0,\dots,X_r]/I}$ is flat over $A$ for $t=d$ and $t=d+1$, and ${\operatorname{rank}_AQ_{d+1}=(\operatorname{rank}_AQ_{d})^{\langle d\rangle}}$, then there are two implications. Firstly, the graded component $Q_t$ is flat over $A$ for all $t\ge d$, so $Q$ has a Hilbert polynomial $P(t)$. Secondly, the theorem states that ${P(t+1)=\operatorname{rank}_AQ_{t+1}=(\operatorname{rank}_AQ_{t})^{\langle t\rangle}}$ for all $t\ge d$.  

We have here used Macaulay representations to describe the assumption on the rank in Gotzmann's persistence theorem, see \cite[Section~4.2]{MR1251956}. It basically says that the rank has a certain polynomial growth from degree $d$ to degree $d+1$. Thus, the theorem implies that if the Hilbert function of $Q$ behaves like a polynomial in two consecutive degrees, then the homogeneous components in all higher degrees are flat and the Hilbert function $h(t)=\operatorname{rank}_AQ_t$ equals the Hilbert polynomial $P(t)$ for all degrees $t\ge d$. In the case when $\operatorname{rank}_AQ_{d}=n\le d$, we have that $(\operatorname{rank}_AQ_{t})^{\langle t\rangle}=n$ for all $t\ge d$. Thus, Theorem~\ref{thm:a} is a generalization of this result for the case with constant Hilbert polynomial $P(t)=n$. 

Note that Gasharov \cite{MR1468870} has proved a generalization of Gotzmann's persistence theorem to modules in the case of polynomial rings over fields, where the flatness is trivial, and that our result extends this to polynomial rings over arbitrary rings.

Our interest in the result of Theorem~\ref{thm:a} comes from its application to Quot schemes. 
In algebraic geometry, Gotzmann's persistence theorem has been used to find defining equations of Hilbert schemes, see, e.g.,  \cite{MR0480478} and \cite[Appendix~C]{MR1735271}. The Quot scheme is a generalization of both Hilbert schemes and Grassmannians, and is therefore a natural object to study. It was first introduced by Grothendieck who also proved its existence using an embedding into a Grassmannian \cite{MR1611822}. This embedding was however only given abstractly. For the case with constant Hilbert polynomials, Skjelnes proved that the embedding of the Quot scheme of points into a Grassmannian is given by an infinite intersection of closed subschemes defined by certain Fitting ideals \cite{QuotFitting}. Moreover, Skjelnes mentions that proving a generalization of Gotzmann's persistence theorem to modules would also prove that only one of those closed subschemes suffices to describe the embedding. We make this statement precise by showing the following consequence of Theorem~\ref{thm:a}.

\begin{thmx}\label{thm:b}
Let $V$ be a projective and finitely generated module over a noetherian ring~$A$. Let $\fO_{\P(V)}^{\oplus p}$ denote the free sheaf of rank $p$ on $f\colon\P(V)\to\operatorname{Spec}(A)=S$. Fix two integers $n\le d$, and let $g\colon G\to S$ denote the Grassmannian scheme parametrizing locally free rank~$n$ quotients of $f_\ast\fO_{\P(V)}^{\oplus p}(d)$. 
We let
\[0\longrightarrow\mathcal{R}_d\longrightarrow g^\ast f_\ast\fO_{\P(V)}^{\oplus p}(d)\longrightarrow\mathcal{E}_d\longrightarrow0\]
denote the universal short exact sequence on the Grassmannian $G$, and let $\mathcal{E}_{d+1}$ be the cokernel of the induced map $\mathcal{R}_d\otimes_{\fO_G}g^\ast f_\ast\fO_{\P(V)}(1)\to g^\ast f_\ast\mathcal{O}^{\oplus p}_{\P(V)}(d+1)$. Then, we have that
\[\operatorname{Quot}^n_{\fO_{\P(V)}^{\oplus p}/\P(V)/S}=V\bigl(\Fitt_{n-1}(\mathcal{E}_{d+1})\bigr),\]
where $\operatorname{Quot}^n_{\fO_{\P(V)}^{\oplus p}/\P(V)/S}$ is the Quot scheme parametrizing flat quotients of $\fO_{\P(V)}^{\oplus p}$ with constant Hilbert polynomial $P(t)=n$ on the fibers, and $V\bigl(\Fitt_{n-1}(\mathcal{E}_{d+1})\bigr)$ denotes the closed subscheme defined by the $(n-1)$:th Fitting ideal of the sheaf $\mathcal{E}_{d+1}$.
\end{thmx}

Our proof of Theorem~\ref{thm:a} depends on the study of Fitting ideals. In particular, we introduce a technique for calculating these ideals for modules with constant Hilbert polynomial. This technique makes it possible to relate Fitting ideals of different graded components of a quotient. Another approach that could give a less explicit proof of Theorem~\ref{thm:a} would be to combine the results of Gasharov \cite{MR1468870} with the techniques of Gotzmann \cite[Section~1]{MR0480478} after finding an upper bound on the regularity of the associated subsheaves. One such bound is given by the Gotzmann number studied in \cite{2014arXiv1410.8612D}, but for modules with arbitrary Hilbert polynomial this number can exceed the degree in which the corresponding submodule is generated. An interesting future approach would be to study if the regularity is in fact bounded by the degree, as is the case for ideals.

However, our method does not only give a constructible proof of Theorem~\ref{thm:a} but also gives a computer efficient algorithm for computing these Fitting ideals. This algorithm is implemented in the package \emph{FiniteFittingIdeals} \cite{gssm2} for the computer program Macaulay2~\cite{M2}.

\par\vspace{\baselineskip}
\noindent\textbf{Acknowledgement.} I am thankful to Roy Skjelnes for introducing me to the problem and for sharing his ideas on this work. David Rydh has given invaluable comments and has helped with the presentation of the results. Finally, I thank Aron Wennman for our discussions that helped with this paper.

\section{Fitting ideals of quotients in degree $d+1$}\label{sec:2}\label{sec:jajaja}
Let $A$ be a ring and consider a finitely generated $A$-module $L$ with a free presentation $F_2\overset{\psi}{\to} F_1\to L\to0$, where $F_1$ has rank $r$. Then, the $i$:th Fitting ideal $\Fitt_i(L)$ of $L$ is the ideal generated by the $(r-i)$-minors of the matrix corresponding to $\psi$. This ideal is independent of the choice of presentation. The reason that we are interested in these ideals is due to the following result.
\begin{proposition}[{\cite[Proposition~20.8]{eisen-comalg}}]\label{prop:efitting}
Let $A$ be a local ring and $L$ a finitely generated $A$-module. Then, $L$ is free of constant rank $n$ if and only if $\Fitt_n(L)=A$ and $\Fitt_{n-1}(L)=0$.
\end{proposition}

In this section and the next, we will have the following assumptions. 
\begin{paragraf}\label{setup}
Let $A$ be a local ring, let $S=A[x_0,\dots,x_r]$, and let $M=\bigoplus_{i=1}^p S$. Let $N\subset M$ be a graded submodule, generated in degrees at most $d$, and write $Q=M/N$. Suppose that $Q_d$ is free of {rank~$n$}, with a basis given by the images of the elements of a set of monomials $\Lambda_{d}=\{\mu_1,\dots,\mu_n\}\subset M_d$, where $\Lambda_{d}$ has the following two properties:
\begin{equation}\label{eq:property}\tag{$\ast$}
\begin{cases}\text{if $\mu\in\Lambda_d$, then $x_0$ divides $\mu$,}\\\text{if $x_i$ divides $\mu\in\Lambda_d$, then $\frac{x_0}{x_i}\mu\in\Lambda_d$.}\end{cases}
\end{equation} 
We denote the first of these  $(\ast_1)$, and the second $(\ast_2)$. By a \emph{monomial}, we mean a product $m=x_0^{d_0}\cdots x_r^{d_r}e_i$ of the variables $x_0,\dots,x_r$ in the $i$:th component $S$ of $M$ for $i=1,\dots,p$. 
\end{paragraf}

\begin{remark}
Sets with property \eqref{eq:property} corresponds to modules with simple Hilbert functions. In Gotzmann's original statement of the theorem, see \cite{MR0480478}, he considered an ideal generated by the first number of monomials in a lexicographic order. If we choose a lexicographic order where $x_0$ is smallest, then the set of the first $n$ monomials in this order has property \eqref{eq:property}.  A set with property \eqref{eq:property} is also related to the notion of \emph{Gotzmann sets}, which has been used to give a combinatorial proof of Gotzmann's persistence theorem for monomial ideals \cite{MR2368639}.
\end{remark}

Our goal of this section will be to study the $(n-1)$:th Fitting ideal of $Q_{d+1}$. 
Write $\Lambda_d^{\mathsf{c}}$ for the set of monomials in $M_{d}\setminus\Lambda_d$. Then, the kernel $N_d$ of $M_d\to Q_d$ is free with a basis consisting of elements 
\begin{equation*}
\alpha_m=m-\sum_{\nu=1}^na_{\nu}(m)\mu_\nu,
\end{equation*} for some $a_{\nu}(m)\in A$, for all monomials $m\in \Lambda_d^{\mathsf{c}}$. In the sequel, we will write $a_\nu(m)=a_\nu$. Let $\Phi_{d+1}=\{(i,m)\mid m\in\Lambda_d^{\mathsf{c}},\ i=0,\dots,r\}$, and let $A^{\Phi_{d+1}}$ be the free module on $\Phi_{d+1}$. 
\begin{lemma}\label{lem_12345}
With the assumptions of {{\ref{setup}}}, the module $Q_{d+1}$ has a free presentation
\[\xymatrix{A^{\Phi_{d+1}}\ar[r]^{\phi_{d+1}}&M_{d+1}\ar[r]&Q_{d+1}\ar[r]&0,}\]
where $\phi_{d+1}(e_{i,m})=x_i\alpha_m$ for every $(i,m)\in\Phi_{d+1}$.
\end{lemma}
\begin{proof}
The module $S_1\otimes_AN_d$ is free with the basis $\{x_i\otimes\alpha_m\mid m\in\Lambda_d^{\mathsf{c}}, i=0,\dots,r\}$ and, since $N$ is generated in degress at most $d$, gives a free presentation of $Q_{d+1}$ as
\[\xymatrix{S_1\otimes_AN_d\ar[r]& M_{d+1}\ar[r]&Q_{d+1}\ar[r]&0.}\]
The canonical isomorphism $S_1\otimes_AN_d\cong A^{\Phi_{d+1}}$ gives the map $\phi_{d+1}$.
\end{proof}
The Fitting ideals of $Q_{d+1}$ are therefore generated by minors of given sizes of the matrix corresponding to $\phi_{d+1}$. As Fitting ideals are independent of the choice of presentation, we will now do a particular change of basis of $A^{\Phi_{d+1}}$, giving a new presentation of $Q_{d+1}$, which simplifies the calculation of the Fitting ideal. 
For any monomial $m''\in M_{d+1}$, we define $\iota(m'')=i$, where $i$ is the smallest integer such that $x_i$ divides $m''$. 
We also define 
\[\Omega_{d+1}=\{(i,m)\mid m\in\Lambda_d^{\mathsf{c}}, \iota(x_im)=i\}\subset \Phi_{d+1}\]
and $\Psi_{d+1}=\Phi_{d+1}\setminus\Omega_{d+1}$. Then there is a canonical isomorphism $A^{\Phi_{d+1}}\cong A^{\Omega_{d+1}}\oplus A^{\Psi_{d+1}}$. Furthermore, we let $\Lambda_{d+1}=\{x_0\mu_1,\dots,x_0\mu_n\}$.
\begin{proposition}\label{prop:dirsumma}
With the assumptions of {{\ref{setup}}}. Let $\langle\Omega_{d+1}\rangle=\phi_{d+1}\bigl(A^{\Omega_{d+1}}\bigr)\subseteq M_{d+1}$, and let $\langle \Lambda_{d+1}\rangle\subseteq M_{d+1}$ denote the submodule generated by $\Lambda_{d+1}$. Then, the free $A$-module $M_{d+1}$ splits as 
\[M_{d+1}=\langle\Omega_{d+1}\rangle\oplus \langle\Lambda_{d+1}\rangle.\]
\end{proposition}
\begin{proof}
By property $(\ast_2)$, we have that any monomial of $M_{d+1}$ lies in precisely one of the sets $\Lambda_{d+1}$ and $\{x_im\mid (i,m)\in\Omega_{d+1}\}$. The submodule $\langle\Omega_{d+1}\rangle$ is generated by the elements $x_i\alpha_m=x_im-\sum_{\nu=1}^na_\nu x_i\mu_\nu$, where $a_\nu\in A$, for all $(i,m)\in\Omega_{d+1}$. If $i=0$, then we define $q_{0,m}=x_0\alpha_m$. On the other hand, if $i\neq0$, then, by property $(\ast_1)$, any monomial $x_i\mu_\nu$ of $x_i\alpha_m$ is of the form $x_0m_\nu$ where $m_\nu=\frac{x_i\mu_\nu}{x_0}$. By subtracting the terms where $m_{\nu}\not\in\Lambda_{d}$, 
we define \[q_{i,m}=x_i\alpha_m-\!\!\!\sum_{m_\nu\in\Lambda_{d}^{\mathsf{c}}}\!\!\!a_{\nu} x_0\alpha_{m_\nu}=x_im-\sum_{\nu=1}^nb_\nu x_0\mu_\nu,\]
where $b_\nu\in A$. 
Thus, the image $\langle\Omega_{d+1}\rangle$ is free with basis $q_{i,m}=x_im-\sum_{\nu=1}^nb_\nu x_0\mu_\nu$ for all $(i,m)\in\Omega_{d+1}$. Now, as any monomial of $M_{d+1}$ either lies in $\Lambda_{d+1}$ or is equal to a difference of $q_{i,m}$ and a linear combination of monomials in $\Lambda_{d+1}$, the result follows.
\end{proof}
Letting $\pi_2\colon\langle\Omega_{d+1}\rangle\oplus\langle\Lambda_{d+1}\rangle \to\langle\Lambda_{d+1}\rangle$ denote the projection, we get the following immediate consequence.
\begin{corollary}\label{cor:dplus1}
The composition $\psi_{d+1}=\pi_2\circ \phi_{d+1}\colon A^{\Psi_{d+1}}\to\langle\Lambda_{d+1}\rangle$ gives a presentation of $Q_{d+1}$ as
\vspace{-4pt}\[\xymatrix{
A^{\Psi_{d+1}}\ar[r]^-{\psi_{d+1}}&\langle\Lambda_{d+1}\rangle\ar[r]& Q_{d+1}\ar[r]&0.}\]
\end{corollary}
By Proposition~\ref{prop:efitting} and the fact that $\langle\Lambda_{d+1}\rangle$ is free of rank $n$, we get the following.
\begin{corollary}
The $A$-module $Q_{d+1}$ is free of rank $n$ if and only if ${\psi_{d+1}\colon A^{\Psi_{d+1}}\to \langle\Lambda_{d+1}\rangle}$ is the zero map. Hence, the $(n-1)$:th Fitting ideal is generated by the entries of the matrix corresponding to $\psi_{d+1}$. 
\end{corollary}

\begin{paragraf}\label{sec:testas} We will now describe a technique for computing the $(n-1)$:th Fitting ideal of $Q_{d+1}$ by  explicitly writing down the sums $\psi_{d+1}(e_{j,m'})$ for all $(j,m')\in\Psi_{d+1}$. Consider a pair $(j,m')\in\Psi_{d+1}$. Let $\iota(x_jm')=i< j$ and write $m=\frac{x_jm'}{x_i}\in\Lambda_d^\mathsf{c}$, so that $(i,m)\in\Omega_{d+1}$. Then we define 
\begin{align*}\beta_{j,m'}=x_j\alpha_{m'}-x_i\alpha_{m}&=x_j\left(m'-\sum_{\nu=1}^na_\nu\mu_\nu\right)-x_i\left(m-\sum_{\nu=1}^nb_\nu\mu_\nu\right)=\\&=\sum_{\nu=1}^nb_\nu x_i\mu_\nu-\sum_{\nu=1}^n a_\nu x_j\mu_\nu.
\end{align*}
 By property $(\ast_1)$, every monomial in this sum
is either of the form $x_0\mu_\nu$ for some $\nu$, or of the form $x_0m''$ for some $m''\in\Lambda_d^{\mathsf{c}}$. 
Thus,
\[\beta_{j,m'}=
\sum_{\nu=1}^nc_\nu x_0\mu_\nu+\sum_{m''\in\Lambda_d^{\mathsf{c}}} c_{m''}x_0m'',\]
for some $c_\nu,c_{m''}\in A$.
Therefore, we write $\sigma_{j,m'}=\sum_{m''\in\Lambda_d^{\mathsf{c}}} c_{m''}x_0\alpha_{m''}$, and define
\[\gamma_{j,m'}=\beta_{j,m'}-\sigma_{j,m'}=
\sum_{\nu=1}^nf_\nu^{(j,m')}x_0\mu_\nu,\]
where $f_\nu^{(j,m')}\in A$. It follows that $\psi_{d+1}(e_{j,m'})=\gamma_{j,m'}$ for all $(j,m')\in\Psi_{d+1}$.
Hence, we have proved the following.
\begin{corollary}\label{cor:alsogood}
The ideal $\Fitt_{n-1}(Q_{d+1})$ is generated by all the coefficients $f_\nu^{(j,m')}$ of $\gamma_{j,m'}$ for all $(j,m')\in\Psi_{d+1}$.
\end{corollary} 
\end{paragraf}

\begin{ex}
Let $S=A[x,y]$ and let $M=Se_1\oplus Se_2$. Assume that $N\subseteq M$ is a graded $S$-submodule generated in degree $d=2$, and write $Q=M/N$. Suppose that $Q_2$ is free of rank~$n=2$ and that we can choose a basis of $Q_2$ such that the quotient $M_2\to Q_2$ 
is given by the matrix
\[\left(\begin{matrix}1&a_2&a_3&0&a_5&a_6\\0&b_2&b_3&1&b_5&b_6\end{matrix}\right),\]
where the columns correspond to the basis elements $x^2e_1,xye_1,y^2e_1,x^2e_2,xye_2,y^2e_2$ of $M_2$. Thus, with the notation of this section, we have that $\Lambda_2=\{x^2e_1,x^2e_2\}$. Then, the inclusion $N_2\hookrightarrow M_2$ is given by the matrix
\[\left(\begin{matrix}
\llap{{$x^2$}\hspace{15pt}}a_2&a_3&a_5&a_6\\
\llap{{$xy$}\hspace{15pt}}-1&0&0&0\\
\llap{{$y^2$}\hspace{20pt}}0&-1&0&0\\
\hline
\llap{{$x^2$}\hspace{18pt}}b_2&b_3&b_5&b_6\\
\llap{{$xy$}\hspace{20pt}}0&0&-1&0\\
\llap{{$y^2$}\hspace{12pt}}\blap{\shortstack{0\\\\\\$xye_1$}}& \blap{\shortstack{0\\\\\\$\smash{y^2}e_1$}}& \blap{\shortstack{0\\\\\\$xye_2$}}&\blap{\shortstack{$-1$\\\\\\$\smash{y^2}e_2$}}
\end{matrix}\right)\]
\ \\\\
\noindent where we have drawn a line to illustrate the two components of $M_2$. The rows above the line correspond to the monomials in the first component $Se_1$ of $M$ and the rows below the line correspond to the monomials of the second component $Se_2$.
Then, we consider the module $Q_3$ as the cokernel of the map $S_1\otimes_AN_2\to M_3$ given by the matrix
\setlength{\arraycolsep}{4pt}
\[\ \ \ \ \left(\begin{matrix}
\llap{{$x^3$}\hspace{25pt}}a_2	&0	&a_3	&0	&a_5	&0	&a_6	&0\\
\llap{{$x^2y$}\hspace{22pt}}-1	&a_2	&0	&a_3	&0	&a_5	&0	&a_6\\
\llap{{$xy^2$}\hspace{28pt}}0	&-1	&-1	&0	&0	&0	&0	&0\\
\llap{{$y^3$}\hspace{28pt}}0	&0	&0	&-1	&0	&0	&0	&0\\
\hline
\llap{{$x^3$}\hspace{26pt}}b_2	&0	&b_3&0	&b_5&0	&b_6&0\\
\llap{{$x^2y$}\hspace{29pt}}0	&b_2&0	&b_3&-1	&b_5&0	&b_6\\
\llap{{$xy^2$}\hspace{27pt}}0	&0	&0	&0	&0	&-1	&-1	&0\\
\llap{{$y^3$}\hspace{10pt}}\bblap{\shortstack{0\\\\\\$x\otimes xye_1$}}	&\blap{\shortstack{0\\\\\\$y\otimes xye_1$}}	&\blap{\shortstack{0\\\\\\$x\otimes \smash{y^2}e_1$}}	&\blap{\shortstack{0\\\\\\$y\otimes \smash{y^2e_1}$}}	&\blap{\shortstack{0\\\\\\$x\otimes xye_2$}}	&\blap{\shortstack{0\\\\\\$y\otimes xye_2$}}	&\blap{\shortstack{0\\\\\\$x\otimes \smash{y^2}e_2$}}	&\blap{\shortstack{$-1$\\\\\\$y\otimes \smash{y^2}e_2$}}
\end{matrix}\right).\]
\ \\\\
\noindent The algorithm that we presented to calculate the Fitting ideals corresponds to column reductions in this matrix: If we have two $(-1)$:s in the same row, we delete the one to the left, since such a $(-1)$ corresponds to an element of $\Psi_3$. After that, we use the other columns to reduce the column without a $(-1)$ to have zeros in every row except the rows corresponding to $\Lambda_3=\{x^3e_1,x^3e_2\}$. Doing these reductions gives the new matrix
\[\left(\begin{matrix}
a_2	&-a_3+a_5b_2+a_2^2	&a_3					&0	&a_5	&-a_6+a_5b_5+a_2a_5&a_6&0\\
-1	&0					&0					&a_3	&0	&0			       	     &0&a_6\\
0	&0					&-1					&0	&0	&0				     &0&0\\
0	&0					&0					&-1	&0	&0				     &0&0\\
\hline
b_2	&-b_3+b_2b_5+a_2b_2	&b_3				&0	&b_5&-b_6+b_5^2+b_2a_5 &b_6&0\\
0	&0					&0					&b_3&-1	&0				     &0&b_6\\
0	&0					&0					&0	&0	&0				     &-1&0\\
0	&0					&0					&0	&0	&0				     &0&-1
\end{matrix}\right).\]
From here, we can immediately read of the generators of $\Fitt_{1}\bigl(Q_3\bigr)$ as 
\[-a_3+a_5b_2+a_2^2,\ -b_3+b_2b_5+a_2b_2,\ -a_6+a_5b_5+a_2a_5,\ -b_6+b_5^2+b_2a_5.\exendhere\]
\end{ex}
\section{Fitting ideals of quotients in higher degree}\label{sec:oomtyckt}
With the notation of the previous section, we will now study the $(n-1)$:th Fitting ideal of $Q_{d+s}$ for $s\ge 1$.
\begin{definition}
For $s\ge1$, we let $\mathcal{I}_{s}$ be the set of ordered $s$-tuples $\i=(i_1,\dots,i_s)$ such that $0\le i_1\le \dots\le i_s\le r$. We give $\mathcal{I}_{s}$ the lexicographical order. Given $\i\in \mathcal{I}_{s}$, we  write $\x_\i=x_{i_1}\cdots x_{i_s}\in S=A[x_0,\dots,x_r]$.
\end{definition}
Let $\Phi_{d+s}=\{(\i,m)\mid \i\in \mathcal{I}_{s},m\in\Lambda_d^{\mathsf{c}}\}$, and let $A^{\Phi_{d+s}}$ be the free module on $\Phi_{d+s}$.
\begin{lemma}\label{lem_1234567}
With the assumptions of {{\ref{setup}}}, the module $Q_{d+s}$ has a free presentation
\[\xymatrix{A^{\Phi_{d+s}}\ar[r]^{\phi_{d+s}}&M_{d+s}\ar[r]&Q_{d+s}\ar[r]&0,}\]
where $\phi_{d+s}(e_{\i,m})=\x_\i\alpha_m$ for every $(\i,m)\in\Phi_{d+s}$.
\end{lemma}
\begin{proof}
The module $S_s\otimes N_d$ is free with basis elements $\x_\i\otimes\alpha_m$, where $\mathbf{i}=(i_1,\dots,i_s)\in \mathcal{I}_s$ and $m\in\Lambda_d^{\mathsf{c}}$. 
The quotient $Q_{d+s}$ is the cokernel of the induced map 
$S_s\otimes N_d\to M_{d+s}$. 
From the canonical isomorphism $A^{\Phi_{d+s}}\cong S_s\otimes_AN_d$ we get the desired presentation.
\end{proof}
For any monomial $m''\in M_{d+s}$, we define $\iota_s(m'')=\mathbf{i}=(i_1,\dots,i_s)$ where $i_1\le\dots\le i_s$ are the minimal integers such that $\mathbf{x}_{\mathbf{i}}$ divides $m''$. Let 
\[\Omega_{d+s}=\{(\i,m)\mid m\in\Lambda_{d}^{\mathsf{c}}, \iota_s(\x_{\i}m)=\i\}\subseteq\Phi_{d+s}\]
and $\Psi_{d+s}=\Phi_{d+s}\setminus\Omega_{d+s}$. Also, we let $\Lambda_{d+s}=\{x_0^s\mu_1,\dots,x_0^s\mu_n\}$.
\begin{proposition}\label{prop:rerer}
Take the assumptions of {{\ref{setup}}}. Let $\langle\Omega_{d+s}\rangle=\phi_{d+s}\bigl(A^{\Omega_{d+s}}\bigr)\subseteq M_{d+s}$ and let $\langle \Lambda_{d+s}\rangle\subseteq M_{d+s}$ denote the submodule generated by $\Lambda_{d+s}$. Then, the free $A$-module $M_{d+s}$ splits as 
\[M_{d+s}=\langle\Omega_{d+s}\rangle\oplus \langle\Lambda_{d+s}\rangle.\]
\end{proposition}
\begin{proof}
By property $(\ast_2)$, any monomial of $M_{d+s}$ lies in precisely one of the sets $\Lambda_{d+s}$ and $\{\x_\i m\mid (\i,m)\in\Omega_{d+s}\}$. 
The submodule $\langle\Omega_{d+s}\rangle$ is free with basis $\x_\i\alpha_m$ for $m\in\Lambda_{d}^\mathsf{c}$ and $\iota_s(\x_\i m)=\i$. We introduce the preorder on the basis elements of $\langle \Omega_{d+s}\rangle$ defined by ${\x_{\i_1}\alpha_{m_1}<\x_{\i_2}\alpha_{m_2}}$ if $\i_1<\i_2$. 
 
Given $(\i,m)\in\Omega_{d+s}$, we write
\[\x_\i\alpha_m=\x_\i m-\Sigma_{\i,m},\]
where $\Sigma_{\i,m}=\sum_{\nu=1}^n a_\nu\x_\i\mu_\nu$ for some $a_\nu\in A$.
 Consider the monomials in $\Sigma_{\i,m}$. Suppose $\x_\i\mu_\nu$ is a monomial of $\Sigma_{\i,m}$ that is not an element of $\Lambda_{d+s}$. Let $\i_\nu=\iota_s(\x_i\mu_\nu)$, and let $m_\nu=\frac{\x_\i\mu_\nu}{\x_{\i_\nu}}\in\Lambda_{d}^\mathsf{c}$. Note that $\i_\nu<\i$ so $\x_{\i_\nu}\alpha_{m_\nu}<\x_\i\alpha_m$. Furthermore, every monomial of
$\x_{\i}\mu_\nu-\x_{\i_\nu}\alpha_{m_\nu}$
is of the form $\x_{\i_\nu}\mu_\xi$ with $\iota_s(\x_{\i_\nu}\mu_\xi)<\i$ for $\xi=1,\dots,n$.
Write \[\x_\i\alpha_m-a_\nu\x_{\i_\nu}\alpha_{m_\nu}=\x_\i m-\Sigma_{\i,m}'.\]
By construction we have that the $\Sigma_{\i,m}'$ is given by replacing the monomial $\x_{\i}\mu_\nu$ in $\Sigma_{\i,m}$ with a linear combination of elements $\x_{\i_\nu}\mu_\xi$ with $\i_\nu<\i$. Thus, we can iteratively remove the monomials of $\Sigma_{\i,m}$ that are not in $\Lambda_{d+s}$  by subtracting elements of the form $a_\nu\x_{\i_\nu}\alpha_{m_\nu}$ that are smaller in our preorder. By property $(\ast_1)$, this procedure will terminate with an element of the form $q_{\i,m}=\x_\i m-\sum_{\nu=1}^n b_\nu x_0^s\mu_\nu$, for some $b_\nu\in A$. 

Using the preorder on the basis elements of $\langle\Omega_{d+s}\rangle$ and iteratively replacing the largest basis elements $\x_\i\alpha_m$ by $q_{\i,m}$, we obtain a change of basis of $\langle \Omega_{d+s}\rangle$. Thus, $\langle \Omega_{d+s}\rangle$ is free with a basis consisting of elements of the form \[q_{\i,m}=\x_\i m-\sum_{\nu=1}^n b_\nu x_0^s\mu_\nu,\]
for all monomials $\x_\i m$ in $M_{d+s}\setminus\Lambda_{d+s}$.  
The result now follows.
\end{proof}
\begin{remark}
This proof reduces to the proof of Proposition~\ref{prop:dirsumma} in the case $s=1$.
\end{remark}
Letting $\pi_2\colon\langle\Omega_{d+s}\rangle\oplus\langle\Lambda_{d+s}\rangle \to\langle\Lambda_{d+s}\rangle$ denote the projection, we get the following. 
\begin{corollary}\label{cor:1212121212}
 The composition $\psi_{d+s}=\pi_2\circ \phi_{d+s}\colon A^{\Psi_{d+s}}\to\langle\Lambda_{d+s}\rangle$ 
gives a presentation of $Q_{d+s}$ as
\[\xymatrix{
A^{\Psi_{d+s}}\ar[r]^-{\psi_{d+s}}& \langle\Lambda_{d+s}\rangle\ar[r]& Q_{d+s}\ar[r]&0.}\]
\end{corollary}
From this result, and the fact that $\langle\Lambda_{d+s}\rangle$ is free of rank $n$, we get a simple description of the $(n-1)$:th Fitting ideal of $Q_{d+s}$.
\begin{corollary}\label{cor:2323}
The $A$-module $Q_{d+s}$ is free of rank $n$ if and only if ${\psi_{d+s}\colon A^{\Psi_{d+s}}\to \langle \Lambda_{d+s}\rangle}$ is the zero map. Hence, the $(n-1)$:th Fitting ideal is generated by the entries of the matrix corresponding to $\psi_{d+s}$. 
\end{corollary}
\begin{proposition}\label{pro:lika}
With the assumptions of {{\ref{setup}}}, there is an equality of ideals
\[{\Fitt_{n-1}(Q_{d+1})=\Fitt_{n-1}(Q_{d+s})}\] for all $s\ge1$. Furthermore, $\Fitt_{N}(Q_{d+s})=A$ for all $N\ge n$ and $s\ge1$.
\end{proposition}
\begin{proof}
The second part follows immediately from Corollary~\ref{cor:1212121212} since $\langle\Lambda_{d+s}\rangle$ is free of rank~$n$. 
For the first part we have, by Corollary~\ref{cor:2323}, that the $(n-1)$:th Fitting ideal of $Q_{d+s}$ is generated by the entries of the matrix corresponding to $\psi_{d+s}=\pi_2\circ\phi_{d+s}$. We will now do certain column reductions of this matrix to help us relate the Fitting ideals.

For any $(\j,m')\in\Psi_{d+s}$, with $\j=(j_1,\dots,j_s)$, we define $i=\iota(j_sm')<j_s$. Then, we let $\j'$ be the ordered $s$-tuple where we remove $j_s$ from $\j$ and replace it with $i$. Letting $m=\frac{\x_\j m'}{\x_{\j'}}$, we have that $\x_\j m'=\x_{\j'}m$ with $\j>\j'$ and $(\j',m)\in\Omega_{d+s}\cup\Psi_{d+s}$. 
Writing $\x_{\j\setminus j_s}=x_1\cdots x_{j_{s-1}}$, and  
using the notation of \ref{sec:testas}, we consider the difference
\begin{align*}
\beta_{\j,m'}=\x_\j\alpha_{m'}-\x_{\j'}\alpha_{m}&=\x_{\j\setminus j_s}(x_{j_s}\alpha_{m'}-x_{i}\alpha_{m})=\x_{\j\setminus j_s}\beta_{j_{s},m'}=\\&=\x_{\j\setminus j_s}\sum_{\nu=1}^nc_\nu x_0\mu_\nu+\x_{\j\setminus j_s}\sum_{m''\in\Lambda_d^{\mathsf{c}}} c_{m''}x_0m'',
\end{align*}
where $c_\nu,c_{m''}\in A$.
Then, we let $\sigma_{j_s,m'}=\sum_{m''\in\Lambda_d^{\mathsf{c}}} c_{m''}x_0\alpha_{m''}$, and define
\[\gamma_{\j,m'}=\x_{\j\setminus j_s}\beta_{j_{s},m'}-\x_{\j\setminus j_s}\sigma_{j_{s},m'}=\x_{\j\setminus j_s}\gamma_{j_{s},m'}=\x_{\j\setminus j_s}\sum_{\nu=1}^n f_\nu^{(j_s,m')}x_0\mu_\nu,\]
where $f_{\nu}^{(j_s,m')}\in \Fitt_{n-1}(Q_{d+1})$. 
Every monomial $\x_{\j\setminus j_s}x_0\mu_\nu$ in $\gamma_{\j,m'}$ can, by Proposition~\ref{prop:rerer}, be expressed uniquely as a sum $\omega_{\nu}+\lambda_\nu$, where $\omega_\nu\in\langle\Omega_{d+s}\rangle$ and $\lambda_{\nu}\in\langle\Lambda_{d+s}\rangle$. 
Hence, 
\[\x_{\j\setminus j_s}\gamma_{j_{s},m'} =\sum_{\nu=1}^n f^{(j_s,m')}_\nu(\omega_\nu+\lambda_\nu),\]
and $\pi_2(\x_{\j\setminus j_s}\gamma_{j_{s},m'})=\sum_{\nu=1}^n f^{(j_s,m')}_\nu\lambda_\nu$.

Now, consider the preorder on the elements $\x_{\i}\alpha_m$ defined by $\x_{\i_1}\alpha_{m_1}<\x_{\i_2}\alpha_{m_2}$ if $\i_1<\i_2$. Then, we see that $\gamma_{\j,m'}$ was constructed from starting with $\x_\j\alpha_{m'}$ and subtracting linear combinations of elements of smaller order. As $\pi_2(\x_\i\alpha_m)=0$ for any $(\i,m)\in\Omega_{d+s}$, it follows that substituting the element $\x_\j\alpha_{m'}$ for $\gamma_{\j,m'}$ is equivalent to column reductions in the matrix corresponding to $\psi_{d+s}=\pi_2\circ\phi_{d+s}$. Thus, we get a new presentation 
\[\xymatrix{A^{\Psi_{d+s}}\ar[r]^-{\psi'_{d+s}}& \langle\Lambda_{d+s}\rangle\ar[r]&Q_{d+s}\ar[r]&0,}\]
where $\psi'_{d+s}(e_{\j,m'})=\pi_2(\gamma_{\j,m'})=\sum_{\nu=1}^n f^{(j_s,m')}_\nu\lambda_\nu$ for all $(\j,m')\in\Psi_{d+s}$. Thus, the entries of the matrix corresponding to $\psi_{d+s}'$ are linear combinations of elements of $\Fitt_{n-1}(Q_{d+1})$. As the matrix of $\psi'_{d+s}$ was obtained by column reductions in the matrix of $\psi_{d+s}$, we have that the ideal $\Fitt_{n-1}(Q_{d+s})$ is generated by the entries of the matrix corresponding to $\psi_{d+s}'$. Hence, it follows that $\Fitt_{n-1}(Q_{d+s})\subseteq\Fitt_{n-1}(Q_{d+1})$.

On the other hand, the ideal $\Fitt_{n-1}(Q_{d+1})$ is generated by the coefficients of $\gamma_{j,m'}$ for $(j,m')\in\Psi_{d+1}$. The coefficients of $\gamma_{j,m'}$ are trivially obtained as coefficients of $\psi_{d+s}(e_{\j,m'})$ where ${\j=(0,\dots,0,j)}$. Thus, we also have that $\Fitt_{n-1}(Q_{d+1})\subseteq\Fitt_{n-1}(Q_{d+s})$, so the result follows.
\end{proof}

\section{Generalizing the results}\label{sec:1}
In this section, we will explain why  a basis with property~\eqref{eq:property} of \ref{setup} is natural to consider, and use this fact  to generalize the results on Fitting ideals to more general cases.

Consider a polynomial ring $S=k[x_0,\dots,x_r]$ over a field $k$. We let $V_+(x_0)$ denote the hyperplane in $\P^r_k=\proj(S)$ defined by $x_0$. Furthermore, given a module~$Q$ over $S$, we let ${\operatorname{Supp}_+(Q)=\{\mathfrak{p}\in\proj(S)=\P^r_k\mid Q_\mathfrak{p}\neq0\}}$ denote the support of $Q$ in~$\P^r_k$. The subscript $+$ is here used to emphasize that these objects lie in $\P^r_k$, rather than in ${\mathbb{A}^{r+1}_k=\operatorname{Spec}(S)}$.
\begin{proposition}\label{prop:viktiggrej}
Let $k$ be a field, let $S=k[x_0,\dots,x_r]$, and let $M=\bigoplus_{i=1}^pS$. Let ${N\subseteq M}$ be a graded $S$-submodule generated in degrees at most $d$. Suppose that $Q=M/N$ has constant Hilbert polynomial $P(t)=n\le d$ and that $Q_d$ is a $k$-vector space of dimension~$n$. If $V_+(x_0)\cap \operatorname{Supp}_+({Q})=\text{\emph{\O}}$, then there is a set of monomials $\Lambda_d=\{\mu_1,\dots,\mu_n\}$ in $M_d$ such that their images in $Q$ constitute a basis of $Q_d$, and $\Lambda_d$ has property $\eqref{eq:property}$ from \ref{setup}. 
\end{proposition}
\begin{proof}
Consider the open set $D_+(x_0)$ in $\P^r_k$ where $x_0$ is invertible.
By the assumption on~$x_0$ it follows, 
since $P(t)=n$, that the $k$-vector space ${Q_{(x_0)}}$ has dimension $n$. 
With some abuse of notation, we write $M_{(x_0)}=\bigoplus_{i=1}^pk\left[x_1,\dots,x_r\right]e_i$. 
Then, as $Q_{(x_0)}$ is an $n$-dimensional quotient of $M_{(x_0)}$, it is generated by the images of monomials in $M_{(x_0)}$ of degrees at most~$n$. 
Introduce a lexicographical order on $M_{(x_0)}$ such that $x_1<\cdots<x_r<e_1<\cdots <e_p$, and let $B$ be the set of the smallest $n$ monomials in $M_{(x_0)}$ that gives a basis of $Q_{(x_0)}$. By construction, $B$ will have the property, for any monomial $m\in M_{(x_0)}$, that $m\not\in B$ implies that $x_im\not\in B$ for any $i$. 

Now we note that the image of the composition $M_d\to M\to M_{(x_0)}$ is precisely the vector space generated by all monomials of degrees at most $d$ in $M_{(x_0)}$. As $d\ge n$, it follows that the composition $M_d\to M_{(x_0)}\to Q_{(x_0)}$ is surjective,  hence so is the map
$Q_d\to Q_{(x_0)}$. Since $Q_d$ and $Q_{(x_0)}$ are $n$-dimenisonal vector spaces, the map $Q_d\to Q_{(x_0)}$ is even an isomorphism.
Thus, by homogenizing the monomials in $B$ using~$x_0$, we get a set $\Lambda_d$ with property~$(\ast_2)$, that maps to a basis of~$Q_d$.

Finally, if $\mu\in\Lambda_d$ is a degree $d$ monomial not divisible by $x_0$, then it follows from $(\ast_2)$ that $\Lambda_{d}$ must contain at least $d+1$ elements. Since $d+1>d\ge n$, this contradicts the assumption that the images of the elements of $\Lambda_{d}$ should constitute a basis of an $n$-dimensional vector space. Thus, $\Lambda_d$ also has property $(\ast_1)$.
\end{proof}

\begin{example}
Let $S=\C[X,Y,Z]$ and let $I=(XY,Z)$. The graded component of the quotient $Q=S/I$ in degree $2$ is free with basis $X^2$ and $Y^2$. 
This corresponds to two points in $\P^2$, namely $(1:0:0)$ and $(0:1:0)$.
The linear form $x=X+Y$ defines a hyperplane of $\P^2$ that contains none of these points. We do a change of variables of $S$ to $x=X+Y,y=Y$ and $z=Z$. Then, $I=(y^2-xy,z)$ and we see that $\mu_1=x^2$ and $\mu_2=xy$ is a basis of $Q_2$ corresponding to a set with property~\eqref{eq:property}.
\end{example}

\begin{lemma}\label{lem:ff}
Let $A$ be a ring and let $I,J\subset A$ be two ideals. Consider a faithfully flat ring homomorphism $A\to A'$. If $I\otimes_AA'=J\otimes_AA'$, then $I=J$. 
\end{lemma}
\begin{proof}
Consider the $A$-module $(I+J)/I$. As $A\to A'$ is flat, we have that \[\big((I+J)/I\big)\otimes_AA'=(I\otimes_AA'+J\otimes_AA')/(I\otimes_AA')=(I\otimes_AA')/(I\otimes_AA')=0.\] 
By faithful flatness, see, e.g., Theorem~7.2 in \cite{MR1011461}, it follows that $(I+J)/I=0$. Therefore, $J\subseteq I$, and, by symmetry, $I=J$. 
\end{proof}

\begin{lemma}\label{lem:r4r4}
Let $A$ be a local ring with residue field $k$. Given a finite field extension $k\hookrightarrow k'$, such that $k'$ is generated by one element over $k$, then there is a ring $A'$ lifting this extension such that $A\to A'$ is faithfully flat. 
\end{lemma}
\begin{proof}
Let $k'=k[x]/(x^n+a_1x^{n-1}+\dots+a_n)$ for some $n\ge2$ and $a_i\in k$. Let $b_i\in A$ be a lift of $a_i$ and write $A'=A[x]/(x^n+b_1x^{n-1}+\dots+b_n)$. Then, $A'$ is a free $A$-module so the result follows.
\end{proof}

We are now ready to prove our main theorem.
\begin{theorem}\label{thm:bigone}
Let $A$ be a ring, let $S=A[X_0,\dots,X_r]$, let $M=\bigoplus_{i=1}^pS$ and let $N$ be a graded $S$-submodule of $M$, generated in degrees at most $d$. Write $Q=M/N$ and suppose that the fibers $Q\otimes_Ak(\mathfrak{p})$ have constant Hilbert polynomial $P(t)=n\le d$, for every prime $\mathfrak{p}\subset A$. If the $d$:th-graded component $Q_d$ is locally free of rank $n$, then ${\Fitt_{n-1}(Q_{d+s})=\Fitt_{n-1}(Q_{d+1})}$ for all $s>0$. Furthermore, $\Fitt_{N}(Q_{d+s})=A$ for all $N\ge n$ and $s\ge0$.
\end{theorem}
\begin{proof}
As Fitting ideals commutes with localization \cite[Corollary~20.5]{eisen-comalg}, we can assume that $A$ is local with maximal ideal $\mathfrak{m}$, and that $Q_d$ is free of rank $n$. Considering $k=A/\mathfrak{m}$, we have an \mbox{$n$-dimensional} vector space $Q_d\otimes_Ak$. By assumption, the Hilbert polynomial of $Q\otimes_Ak$ is constant, $P(t)=n$. 
Now, we iteratively add elements to $k$, giving a finite field extension $k\hookrightarrow k'$, such that $k'$ has enough elements to find a hyperplane $V_+(x_0)$ that does not intersect the support of ${Q\otimes_Ak'}$ in $\P^r_{k'}$. By Lemma~\ref{lem:r4r4}, this field $k'$ can be lifted to a ring $A'$, giving a faithfully flat map $A\to A'$. Fitting ideals commute with base change, also by \cite[Corollary~20.5]{eisen-comalg}, so ${\Fitt_{n-1}(Q_{d+s})\otimes_AA'=\Fitt_{n-1}(Q_{d+s}\otimes_AA')}$. Thus, since $A\to A'$ is faithfully flat, it follows by Lemma~\ref{lem:ff}, that if ${\Fitt_{n-1}(Q_{d+s}\otimes_AA')=\Fitt_{n-1}(Q_{d+s'}\otimes_AA')}$ for any $s,s'>0$, then $\Fitt_{n-1}(Q_{d+s})=\Fitt_{n-1}(Q_{d+s'})$.

We can therefore assume that $k$ has enough elements to find a hyperplane $V_+(x_0)$ that does not intersect the support of ${Q\otimes_Ak}$ in $\P^r_k$.
After a change of variables of $S_1$ we then, by Proposition~\ref{prop:viktiggrej}, get a set of monomials $\Lambda_d=\{\mu_1,\dots,\mu_n\}\subset M_d$  with property \eqref{eq:property}, giving a basis of $Q_d\otimes_Ak$. 
By Nakayama's lemma, this basis is also a basis of $Q_d$.

Thus, we have reduced the problem to the assumptions of {\ref{setup}}, so the result follows from Proposition~\ref{pro:lika}.
\end{proof}

\section{A generalization of Gotzmann's persistence theorem}
In the case of polynomial rings over fields, Gasharov, in \cite{MR1468870}, proved a generalization of Gotzmann's persistence theorem to modules. When working over fields, the flatness condition is immediate, but we will use his result to prove Theorem~\ref{thm:a} over arbitrary rings. We state here a simplified version of one of his main theorems.
\begin{theorem}[{\cite[Theorem~4.2]{MR1468870}}]\label{thm:gash}
Let $k$ be a field, let $S=k[X_0,\ldots,X_r]$, and let ${M=\bigoplus_{i=1}^pS}$ be a free $S$-module of rank $p$. Let $N\subseteq M$ be a submodule, generated in degrees at most $d$, and consider the quotient $Q=M/N$. 
If $\dim_kQ_{d+1}=(\dim_kQ_d)^{\langle d\rangle}$, then $\dim_kQ_{d+2}=(\dim_kQ_{d+1})^{\langle d+1\rangle}$.
\end{theorem}
This theorem uses the notation of Macaulay representations defined as follows: for any integer $d$, the $d$:th Macaulay representation of an integer $a$ is 
\[a=\binom{m_d}{d}+\binom{m_{d-1}}{d-1}+\dots+\binom{m_1}{1},\]
where $m_d,\dots,m_1$ are the unique integers satisfying $m_d>m_{d-1}>\dots>m_1\ge0$. Then, we define 
\[a^{\langle d\rangle}=\binom{m_d+1}{d+1}+\binom{m_{d-1}+1}{d-1+1}+\dots+\binom{m_1+1}{1+1},\]
see also \cite[Section~4.2]{MR1251956}.
In the case with constant Hilbert polynomials, this theorem reduces to the following, see also \cite[Remark~2.4]{QuotFitting}.
\begin{corollary}\label{cor:gash}
Let $k$ be a field, let $S=k[X_0,\ldots,X_r]$, and let $M=\bigoplus_{i=1}^pS$ be a free $S$-module of rank $p$. Let $N\subseteq M$ be a submodule, generated in degrees at most $d$, and write $Q=M/N$. If $Q_t$ has dimension $n\le d$ over $k$ for $t=d$ and $t=d+1$, then $Q_t$ has dimension~$n$ for all $t\ge d$.
\end{corollary}
\begin{proof}
If $n\le d$, then the $d$:th Macaulay representation of $n$ is
\[n=\binom{d}{d}+\binom{d-1}{d-1}+\dots+\binom{d-n+1}{d-n+1}.\]
Therefore, $n^{\langle t\rangle}=n$ for all $t\ge d$, and the result follows by Theorem~\ref{thm:gash}.
\end{proof}

We now get the generalized version of Gotzmann's persistence theorem, that we called Theorem~\ref{thm:a} in the introduction, as an immediate consequence.

\begin{corollary}\label{cor:cor}
Let $A$ be a ring, let $S=A[X_0,\dots,X_r]$, let $M=\bigoplus_{i=1}^pS$, and let $N$ be a graded $S$-submodule of $M$, generated in degrees at most $d$. Write $Q=M/N$ and let $n\le d$. If $Q_t$ is locally free of rank $n$ for $t=d$ and $t=d+1$, then $Q_t$ is locally free of rank~$n$ for all~${t\ge d}$.
\end{corollary}
\begin{proof}
By Corollary~\ref{cor:gash}, we have that $Q\otimes_Ak(\mathfrak{p})$ has constant Hilbert polynomial $P(t)=n$ for all prime ideals $\mathfrak{p}\subset A$. Thus, since $Q_{d}$ is locally free of rank $n$ it follows, by Theorem~\ref{thm:bigone}, that $\Fitt_{n}(Q_{d+s})=A$ and \[{\Fitt_{n-1}(Q_{d+s})=\Fitt_{n-1}(Q_{d+1})},\] for all $s>0$. Furthermore, since $Q_{d+1}$ is locally free of rank $n$, we have ${\Fitt_{n-1}(Q_{d+1})=0}$. Hence, $\Fitt_{n-1}(Q_{d+s})=0$ for all $s>0$, so $Q_{d+s}$ is locally free of rank $n$ for all $s\ge0$.
\end{proof}

\begin{remark}
Note that it follows, with the assumptions of Corollary~\ref{cor:cor}, that if $\Lambda_{d}$ gives a basis of $Q_{d}$ with property \eqref{eq:property}, then $\Lambda_{d+s}$ gives a basis of $Q_{d+s}$ for every $s\ge0$.
\end{remark}

\section{The Quot scheme of points}
We will end by showing an application of these results to Quot schemes of points. 
Let $S$ be a scheme and let $\mathcal{F}$ be a coherent sheaf on $\P^r\to S$. The Quot functor $\mathcal{Q}uot^P_{\mathcal{F}/\P^r/S}$ was first introduced by Grothendieck in \cite{MR1611822}. It parametrizes quotient sheaves of $\mathcal{F}$ that are flat over the base with Hilbert polynomial $P(t)$ on the fibers. A classical result is that this functor is representable by a projective scheme, which we denote by $\operatorname{Quot}^P_{\mathcal{F}/\P^r/S}$. This was originally proved using an embedding into a Grassmannian. See also \cite{MR0209285} and \cite{MR2222646}. The existence of this embedding was only proved abstractly, and there exist no easy description of the defining equations of the embedding in general. 

In the special case with constant Hilbert polynomial $P(t)=n$, we get the \emph{Quot scheme of points}, see, e.g., \cite{MR2346502}. In \cite{QuotFitting}, Skjelnes finds the locus that defines the Quot scheme of points as a closed subscheme of a Grassmannian. 
\begin{theorem}[{\cite[Theorem~4.7]{QuotFitting}}]\label{thm:roy1}
Let $V$ be a projective and finitely generated module on a noetherian ring $A$. Let $\fO^{\oplus p}_{\P(V)}$ denote the free sheaf of rank $p$ on $f\colon\P(V)\to\operatorname{Spec}(A)=S$. Fix two integers $n\le d$, and let $g\colon G\to S$ denote the Grassmann scheme parametrizing locally free rank $n$ quotients of $f_\ast\fO(d)^{\oplus p}_{\P(V)}$. We let
\[0\longrightarrow\mathcal{R}_d\longrightarrow g^\ast f_\ast\fO^{\oplus p}_{\P(V)}(d)\longrightarrow\mathcal{E}_d\longrightarrow0\]
denote the universal short exact sequence on the Grassmann $G$, and let $\mathcal{E}_{d+s}$ denote the cokernel of the induced map $\mathcal{R}_d\otimes_{\fO_G}g^\ast f_\ast\fO_{\P(V)}(s)\to g^\ast f_\ast\fO^{\oplus p}_{\P(V)}(d+s)$. Then, we have that
\[\operatorname{Quot}^n_{\fO^{\oplus p}_{\P(V)}/\P(V)/S}=\bigcap_{s\ge1}V\bigl(\Fitt_{n-1}(\mathcal{E}_{d+s})\bigr),\]
where $V\bigl(\Fitt_{n-1}(\mathcal{E}_{d+1})\bigr)$ denotes the closed subscheme defined by the $(n-1)$:th Fitting ideal of $\mathcal{E}_{d+1}$.
\end{theorem}

A consequence of this, using Corollary~\ref{cor:gash}, is that only one Fitting ideal is necessary for describing the underlying topological spaces  
$\bigl|\operatorname{Quot}^n_{\fO^{\oplus p}_{\P(V)}/\P(V)/S}\bigr|=\bigl|V\bigl(\Fitt_{n-1}(\mathcal{E}_{d+1})\bigr)\bigr|,$ see \cite[Corollary~4.8]{QuotFitting}. We can now prove that the whole scheme structure is given by only one Fitting ideal, which was Theorem~\ref{thm:b} of our introduction.
\begin{theorem}
With the notation of Theorem~\ref{thm:roy1}, we have that
\[\operatorname{Quot}^n_{\fO^{\oplus p}_{\P(V)}/\P(V)/S}=V\bigl(\Fitt_{n-1}(\mathcal{E}_{d+1})\bigr).\]
\end{theorem}
\begin{proof}
The result follows directly from applying Corollary~\ref{cor:cor} to Theorem~\ref{thm:roy1}.
\end{proof}

\bibliography{references}{}
\bibliographystyle{amsalpha}

\end{document}